\documentclass{article}%
\usepackage{amsfonts, amsmath, amssymb}
\usepackage{subfig}
\usepackage{pstricks, pst-plot, pst-node, multido}
\usepackage{amsmath}
\usepackage{amsfonts}
\usepackage{amssymb}
\usepackage{graphicx}

\newtheorem{theorem}{Theorem}

\newtheorem{corollary}[theorem]{Corollary}

\newtheorem{lemma}[theorem]{Lemma}

\newtheorem{proposition}[theorem]{Proposition}
\newtheorem{remark}[theorem]{Remark}

\newenvironment{proof}[1][Proof]{\noindent\textbf{#1.} }{\ \rule{0.5em}{0.5em}}
\begin{document}

\title{Global bifurcation of planar and spatial periodic solutions in the
restricted n-body problem}
\author{C. Garc\'{\i}a-Azpeitia
\and J. Ize\\{\small Depto. Matem\'{a}ticas y Mec\'{a}nica, IIMAS-UNAM, FENOMEC, }\\{\small Apdo. Postal 20-726, 01000 M\'{e}xico D.F. }\\{\small cgazpe@hotmail.com}}
\maketitle

\begin{abstract}
The paper deals with the study of a satellite attracted by n primary bodies,
which form a relative equilibrium. We use orthogonal degree to prove global
bifurcation of planar and spatial periodic solutions from the equilibria of
the satellite. In particular, we analyze the restricted three body problem and
the problem of a satellite attracted by the Maxwell's ring relative equilibrium.


\end{abstract}

\section{Introduction}

The restricted $n$-body problem is the study of the movement of a satellite
attracted by $n$ primary bodies which are rotating, at a constant angular
speed, around an axis. Since the mass of the satellite is small, one assumes
that the satellite does not perturb the trajectories of the primaries. We
shall suppose that these trajectories form a relative equilibrium and, as
such, are in a plane, let us say the $(x,y)$-plane. In this paper, the
primaries are assumed to be point masses or, equivalently, homogeneous spheres.

The purpose of this paper is to prove the existence of a global bifurcation of
periodic solutions for the satellite, starting from the relative equilibria of
the satellite. These solutions will form a continuum in the plane of the
primaries and we shall also prove that there are other global branches of
solutions out of that plane. The proof is based on the use of a topological
degree for maps that commute with some symmetries and are orthogonal to the
infinitesimal generators for these symmetries. We give results for a general
situation and applications to some special cases such as the restricted three
body problem and the Maxwell's Saturn ring, that is when there are $n$
primaries, of the same mass, forming a regular polygon, and a central larger
mass, as a classical model for Saturn and one ring around it. However, for the
general result the primaries may have different masses and may be located at
any relative equilibrium.

The study of relative equilibria for the restricted $n$-body setting is a
classical problem and there is a vast literature for it. For instance, in the
case of the restricted three-body problem, the local bifurcation of planar
periodic orbits from the Lagrange points is well known (see \cite{Ma90},
\cite{MeHa91}). There is a huge number of numerical explorations for this
restricted three-body problem, under a variety of hypotheses, such as the
bifurcation near $L_{4}$, where the mass of the primary is the bifurcation
parameter above Routh's number, \cite{Ba02}, \cite{Si10}, with a period
doubling cascade. \cite{SEE00} has a study of the phase space for solutions
near $L_{4}$ and \cite{EF09} treats the elliptic case where one has four
periods for solutions close to $L_{4}$. The stability of the orbits close to
$L_{4}$ is studied in \cite{Ef05} and the connection from $E_{3}$ to $L_{4}$
is explored in \cite{Pi09}. A very complete numerical study, \cite{DoRo07},
using AUTO, shows the many different types of periodic orbits and the
connections between the Lagrange points and also the secondary bifurcations
along the curves in the ${x,y,\mu}$ space, where $\mu$ is the mass of one of
the primaries. From a very applied point of view, one may cite \cite{KaGuKo05}
and \cite{GoLlSi00}.

In the case of the Maxwell ring, besides the theoretical results of
\cite{SiMo71} and \cite{Me99}, one has also many numerical results, such as
\cite{Pi05}, where the author studies numerically some families of solutions
around the central body and around the ring for a low number of peripherals,
with a theoretical approximation for the case of a satellite far from the
ring. A theoretical study, with averaging techniques is given in \cite{LS11}
for orbits far from the set of primaries (comets) and close to one of the
primaries (Hill solutions). Similarly \cite{MK07} proposes a regularization
for collision orbits. Closer to the spirit of the present paper, we mention
some of the more recent papers in the bibliography, in particular
\cite{ArEl04}, \cite{BaEl04}, \cite{Ka08} and \cite{BBS08}, where a numerical
classification of the different types of orbits is done.

The paper which is closer to ours, in the sense that is based on topological
arguments similar to ours and giving global results for the possible
connections between the relative equilibria is \cite{MaRy04}, for the
restricted three-body problem.

A final introductory comment, about topological methods, in particular in
bifurcation problems, may be useful: the degree arguments, coupled with group
representation ideas, give global information, i.e., an indication of where
the bifurcation branches could go. Also, since the results are valid for
problems which are deformation of the original problem, the method does not
require high order computations and they may be applied in some degenerate
cases (for instance it is not necessary that the bifurcation parameter crosses
a critical value with non-zero speed; it enough that it crosses eventually).
However, knowledge of some generic property, like a Morse condition, implies
an easy application of the argument. This may be not the case for problems
with more parameters (see however \cite{I95}). An immediate drawback of this
approach is that topological methods do not provide a detailed information on
the local behavior of the bifurcating branch, such as stability or the
existence of other type of solutions, like KAM tori. Other methods, such as
normal forms or special coordinates, should be used for these purposes, but
they only provide local information near the critical point. In a similar way,
our degree arguments give only partial results on resonances and other tools
should be used. Topological methods provide an interesting complement of information.

\section{Setting the problem}

Newton's equations describing the movement of a satellite, in rotating
coordinates and with angular speed equal to $1$, are%
\begin{align*}
\ddot{x}+2\bar{J}\dot{x} & =\nabla V(x)\text{ with }\\
V(x) & :=\frac{\left\Vert \bar{I}x\right\Vert ^{2}}{2}+\sum_{j=1}^{n}%
m_{j}\phi_{\alpha}(\left\Vert x-(a_{j},0)\right\Vert )\text{,}%
\end{align*}
where $x\in\mathbb{R}^{3}$ is the position of the satellite and $(a_{j},0)$ is
the position of a primary body with mass $m_{j}$. The function $\phi_{\alpha}$
represents the attraction between the bodies, where we suppose that
$\phi_{\alpha}^{\prime}=-1/x^{\alpha}$, and we include the gravitational
potential for $\alpha=2$. The matrices $\bar{I}$ and $\bar{J}$ are defined by
\[
\bar{I}=\mathrm{diag}(I,0)\text{ and }\bar{J}=\mathrm{diag}(J,0)\text{,}%
\]
where $J$ and $I$ are the symplectic and identity $2\times2$ matrices.

Here we assume that the primary bodies form a relative equilibrium. Because of
the homogeneity of the potential, we may rescale the space so that the angular
velocity is $1$. As all relative equilibria are planar for the $n$-body
problem, thus the positions of the primary bodies $a_{j}\in\mathbb{R}^{2}$
must satisfy the relation%
\[
a_{i}=\sum_{j=1~(j\neq i)}^{n}m_{j}\frac{a_{i}-a_{j}}{\left\Vert a_{i}%
-a_{j}\right\Vert ^{\alpha+1}}\text{.}%
\]

The equilibria of the satellite are just the critical points of the potential
$V$. From the potential we can prove that all equilibria are planar. Now, we
wish to find the Hessian of the potential at a planar equilibrium.

\begin{proposition}
Let $d_{j}$ be the distance between $x_{0}=(x,y,0)$ and the primary body
$(a_{j},0)=(x_{j},y_{j},0)$. The Hessian matrix of the potential is
\[
D^{2}V(x_{0})=\left( I+\sum_{j=1}^{n}m_{j}A_{j},-\sum_{j=1}^{n}m_{j}%
/d_{j}^{\alpha+1}\right) \text{,}%
\]
where the matrices $A_{j}$ are defined by%
\begin{equation}
A_{j}=\frac{(\alpha+1)}{d_{j}^{\alpha+3}}\left(
\begin{array}
[c]{cc}%
(x-x_{j})^{2} & (x-x_{j})(y-y_{j})\\
(x-x_{j})(y-y_{j}) & (y-y_{j})^{2}%
\end{array}
\right) -\frac{I}{d_{j}^{\alpha+1}}\text{.} \label{Eq101}%
\end{equation}

\end{proposition}

\begin{proof}
Since the function $\phi_{\alpha}(d_{j})$ has Hessian
\[
D^{2}\phi_{\alpha}(d_{j})=\frac{\alpha+1}{d_{j}^{\alpha+3}}\left(
\begin{array}
[c]{ccc}%
(x-x_{j})^{2} & (x-x_{j})(y-y_{j}) & 0\\
(x-x_{j})(y-y_{j}) & (y-y_{j})^{2} & 0\\
0 & 0 & 0
\end{array}
\right) -\frac{I}{d_{j}^{\alpha+1}}\text{,}%
\]
hence $D^{2}\phi_{\alpha}(d_{j})=diag(A_{j},-1/d_{j}^{\alpha+1})$. From this
fact we get the Hessian of $V$.
\end{proof}

Now we want to estimate the number of equilibria provided that the potential
is a Morse function, which is more than a reasonable condition. This is a
generic condition, which is met in our applications, but which could not hold
in some cases. As a matter of fact, we only need that the critical points
should be isolated. Because all equilibria are in the plane, we may restrict
the potential to planar points.

\begin{proposition}
Let us assume that the potential of the satellite is in the plane with
$\alpha\in\lbrack1,\infty)$. Then the potential does not have maximum points.
In addition, if the potential is a Morse function, then
\[
\text{\#saddle points}=n-1+\text{\#minimum points.}%
\]
Moreover, since the potential has a global minimum, there are at least $n$
saddle points.
\end{proposition}

\begin{proof}
The potential in the plane has Hessian $D^{2}V(x_{0})=I+\sum_{j=1}^{n}%
m_{j}A_{j}$, and the trace of $D^{2}V(x_{0})$ is%
\begin{equation}
T=2+(\alpha-1)\sum_{j=1}^{n}\frac{m_{j}}{d_{j}^{\alpha+1}}\text{.}
\label{Eq102}%
\end{equation}
Consequently, the potential does not have maximum points as the trace is
positive for $\alpha\in\lbrack1,\infty)$. Moreover, we know that $V(x)$ is
positive and that $V(x)\rightarrow\infty$ as $x\rightarrow\{\infty
,a_{1},...,a_{n}\}$, then $V$ has at least a global minimum in $\Omega$. Since
the gradient of $V(x)$ is dominated by the identity, for large $\left\Vert
x\right\Vert $, the critical points are bounded.

Let us define the set $\Omega$ as a ball of radius $\rho$, minus small balls
of radii $\rho^{-1}$ with centers at $a_{j}$. Since the gradient $\nabla V$
points outward in $\partial\Omega$ provided $\rho$ is big enough, then by the
Poincar\'{e}--Hopf theorem the degree of $\nabla V(x)$ is equal to $1-n$.
Furthermore, since $V(x)$ is a Morse function, that is the critical points are
non-degenerate, then this degree is the sum of the local indices. Each of
these indices is the sign of the determinant of the Hessian matrix, that is
$1$ for a minimum and $-1$ for a saddle point. Then%
\[
1-n=\deg_{\Omega}\nabla V=\text{\#minimum points}-\text{\#saddle points }%
\]

\end{proof}

\section{Bifurcation theorem}

In order to explain our results, we may give a short description of the steps
to prove the bifurcation theorem.

We wish to remark that we follow the ideas from the book \cite{IzVi03}, where
more general bifurcation theorems are proven. In addition, in the thesis
\cite{Ga10} one may find a systematic application to different Hamiltonian
systems and situations.

\subsection{The bifurcation operator}

Our aim is to find bifurcation of periodic solutions from the equilibria of
the satellite. First, we make the change of variables from $t$ to $t/\nu$.
Hence, the $2\pi/\nu$-periodic solutions of the differential equations are the
$2\pi$-periodic solutions of%
\[
-\nu^{2}\ddot{x}-2\nu\bar{J}\dot{x}+\nabla V(x)=0\text{.}%
\]

Let $H_{2\pi}^{2}(\mathbb{R}^{n})$ be the Sobolev space of $2\pi$-periodic
functions, with the corresponding regularity. We define the collision points
as the set $\Psi=\{a_{1},...,a_{n}\}$ and the collision-free paths as the set
\[
H_{2\pi}^{2}(\mathbb{R}^{3}\backslash\Psi)=\{x\in H_{2\pi}^{2}(\mathbb{R}%
^{3}):x(t)\neq a_{j}\}.
\]
Recall that functions in this space are continuous. Hence, we define the
bifurcation operator $f:H_{2\pi}^{2}(\mathbb{R}^{3}\backslash\Psi
)\times\mathbb{R}^{+}\rightarrow L_{2\pi}^{2}$ as%
\[
f(x,\nu)=-\nu^{2}\ddot{x}-2\nu\bar{J}\dot{x}+\nabla V(x)\text{.}%
\]
In view of the definitions, the collision-free $2\pi$-periodic solutions are
zeros of the bifurcation operator $f(x,\nu)$. Furthermore, the operator $f$ is
well defined and continuous.

Now, we define the actions of the group $\mathbb{Z}_{2}\times S^{1}$ on
$H_{2\pi}^{2}(\mathbb{R}^{3}\backslash\Psi)$ as
\[
\rho(\kappa)x=Rx(t)\text{ and }\rho(\varphi)x=x(t+\varphi),
\]
where $R=\mathrm{diag}(1,1,-1)$ is just the reflection which fixes the plane. 

Since $V(x)$ is invariant with respect to the reflection, the gradient $\nabla
V$ is a $\mathbb{Z}_{2}$-equivariant map. Moreover, since the equation is
autonomous and $R$ commutes with the matrix $\bar{J}$, then
\[
f(\rho(\kappa,\varphi)x)=\rho(\kappa,\varphi)f(x)\text{.}%
\]
Therefore $f(x)$ is a $\mathbb{Z}_{2}\times S^{1}$-equivariant map.

Now, the generator of the group $S^{1}$ on the space $H_{2\pi}^{2}$ is
$D_{\varphi}(\rho(\varphi)x)_{\varphi=0}=\dot{x}$. As the operator $f(x)$
satisfies the equality%
\[
\left\langle f(x),\dot{x}\right\rangle _{L_{2\pi}^{2}}=(-\nu^{2}\left\vert
\dot{x}\right\vert ^{2}/2+V(x))|_{0}^{2\pi}=0\text{,}%
\]
then the operator $f(x)$ is orthogonal to the generator $\dot{x}$ in $L_{2\pi
}^{2}$. Given this condition we say that the operator $f(x)$ is a
$\mathbb{Z}_{2}\times S^{1}$-orthogonal map. The orthogonality corresponds to
the conservation of energy.

Finally, since all the equilibria are planar, the isotropy subgroup of an
equilibrium $x_{0}$ is $\mathbb{Z}_{2}\times S^{1}$. This means that all
equilibria are fixed by the action of $\mathbb{Z}_{2}\times S^{1}$.

\subsection{The Lyapunov-Schmidt reduction}

We want to use the orthogonal degree in order to prove bifurcation, but since
this degree is defined only in finite dimensions, we need to reduce the
bifurcation operator to finite dimensions. To achieve this, let us set the
Fourier series of the bifurcation operator as
\[
f(x)=\sum_{l\in\mathbb{Z}}\left( l^{2}\nu^{2}x_{l}-2il\nu\bar{J}x_{l}%
+g_{l}\right) e^{ilt}\text{,}%
\]
where $x_{l}$ and $g_{l}$ are the Fourier modes of $x$ and $\nabla V(x)$,
respectively. Since $l^{2}\nu^{2}I-2il\nu\bar{J}$ is invertible for all $l$'s
except a finite number, we can make a Lyapunov-Schmidt reduction to a finite
space. In fact, we perform a global reduction, using the global implicit
function theorem, with the right bounds taking care of the collision points
$\Psi$.

In that way, we get that the zeros of the bifurcation operator are the zeros
of the bifurcation function
\[
f(x_{1},x_{2}(x_{1},\nu),\nu)=\sum_{\left\vert l\right\vert \leq p}(l^{2}%
\nu^{2}x_{l}-2il\nu\bar{J}x_{l}+g_{l})e^{ilt}\text{,}%
\]
where $x_{1}$ corresponds to the $2p+1$ modes and $x_{2}$ to the complement.

Consequently, the linearized bifurcation function at an equilibrium $x_{0}$ is%
\[
f^{\prime}(x_{0},\nu)x_{1}=\sum_{\left\vert l\right\vert \leq p}\left(
l^{2}\nu^{2}I-2il\nu\bar{J}+D^{2}V(x_{0})\right) x_{l}e^{ilt}\text{.}%
\]

In fact, $\nabla V(x)= D^{2}V(x_{0})(x-x_{0}) +...$, close to $x_{0}$ and the
Fourier components of $x-x_{0}$ are $x_{l}$ for $l\neq0$ and we rename the
stationary mode as $x_{0}$.

So the linearized bifurcation equation has blocks $M(l\nu)$ for $l\in
\{0,...,p\}$, where the block $M(\lambda)$ is%
\[
M(\lambda)=\lambda^{2}I-2i\lambda\bar{J}+D^{2}V(x_{0})\text{.}%
\]

\subsection{Irreducible representations}

In the following part, we analyze the symmetries of the group $\mathbb{Z}%
_{2}\times S^{1}$. Since the action of $(\kappa,\varphi)\in\mathbb{Z}%
_{2}\times S^{1}$ on Fourier modes $e^{ilt}x_{l}$ is
\[
\rho(\kappa,\varphi)(e^{ilt}x_{l})=Re^{il\varphi}(e^{ilt}x_{l}),
\]
then the action on the block $M(l\nu)$ is given by $\rho(\kappa,\varphi
)x_{l}=Re^{il\varphi}x_{l}$.

Now, as the action of $\mathbb{Z}_{2}$ on $\mathbb{C}^{3}$ is $\rho
(\kappa)=diag(1,1,-1)$, the space $\mathbb{C}^{3}$ has two irreducible
representations: $V_{0}=\mathbb{C}^{2}\times\{0\}$ and $V_{1}=\{0\}\times
\mathbb{C}$. That is, the group $\mathbb{Z}_{2}$ acts on $V_{0}$ as
$\rho(\kappa)=1$ and on $V_{1}$ as $\rho(\kappa)=-1$. Hence, by Schur's lemma
we know that the matrix $M(\lambda)$ must satisfy
\[
M(\lambda)=\mathrm{diag}(M_{0}(\lambda),M_{1}(\lambda)).
\]
Actually, from the explicit Hessian $D^{2}V(x_{0})$ we have%
\begin{align}
M_{1}(\lambda) & =\lambda^{2}-\sum_{j=1}^{n}m_{j}/d_{j}^{\alpha+1}\text{
and}\label{Eq201}\\
M_{0}(\lambda) & =\lambda^{2}I-2iJ\lambda+\left( I+\sum_{j=1}^{n}m_{j}%
A_{j}\right) \text{.}\nonumber
\end{align}

Consequently, the action of the group $\mathbb{Z}_{2}\times S^{1}$ on the
block $M_{0}(\nu)$ is $(\kappa,\varphi)x=e^{i\varphi}x$. Therefore the element
$(\kappa,0)$ leaves fixed the points for $M_{0}(\nu)$, so the isotropy
subgroup for $M_{0}(\nu)$ is the one generated by $(\kappa,0)$,%
\[
\mathbb{Z}_{2}=\left\langle (\kappa,0)\right\rangle .
\]
For $M_{1}(\nu)$ the action of the group $\mathbb{Z}_{2}\times S^{1}$ is
$(\kappa,\varphi)x=-e^{i\varphi}x$. It follows that $(\kappa,\pi)$ leaves
fixed the points for $M_{1}(\nu)$, thus the isotropy subgroup for $M_{1}(\nu)$
is generated by $(\kappa,\pi)$,
\[
\mathbb{\tilde{Z}}_{2}=\left\langle (\kappa,\pi)\right\rangle .
\]

\subsection{The orthogonal degree}

The orthogonal degree is defined for orthogonal maps that are non-zero on the
boundary of some open bounded invariant set. The degree is made of integers,
one for each orbit type, and it has all the properties of the usual Brouwer
degree. Hence, if one of the integers is non-zero, then the map has a zero
corresponding to the orbit type of that integer. In addition, the degree is
invariant under orthogonal deformations that are non-zero on the boundary. The
degree has other properties such as sum, products and suspensions, for
instance, the degree of two pieces of the set is the sum of the degrees.

Now, if one has an isolated orbit, then its linearization at one point of the
orbit $x_{0}$ has a block diagonal structure, due to Schur's lemma, where the
isotropy subgroup of $x_{0}$ acts as $\mathbb{Z}_{n}$ or as $S^{1}$.
Therefore, the orthogonal index of the orbit is given by the signs of the
determinants of the submatrices where the action is as $\mathbb{Z}_{n}$, for
$n=1$ and $n=2$, and the Morse indices of the submatrices where the action is
as $S^{1}$. In particular, for problems with a parameter, if the orthogonal
index changes at some value of the parameter, one will have bifurcation of
solutions with the corresponding orbit type. Here, the parameter is the
frequency $\nu$.

Any Fourier mode will give rise to an orbit type (modes which are multiples of
it), hence one has an element of the orthogonal degree for each mode.
Furthermore, if $x(t)$ is a periodic solution, with frequency $\nu$, then
$y(t)= x(nt)$ is a $2\pi/n$-periodic solution, with frequency $\nu/n$. Hence,
any branch arising from the fundamental mode will be reproduced in the
harmonic branch. If one wishes to study period-doubling, then one has to
consider the branch corresponding to $\pi$-periodic solutions.

The complete study of the orthogonal degree theory is given in \cite{IzVi03}.

\begin{theorem}
Supposing that the matrix $M(0)=D^{2}V(x_{0})$ is invertible, we define
\begin{equation}
\eta_{k}(\lambda)=\sigma(n_{k}(\lambda-\rho)-n_{k}(\lambda+\rho))\text{,}%
\end{equation}
where $\sigma=\mathrm{sgn}(\det M_{0}(0))$ and $n_{k}(\lambda)$ is the Morse
index of $M_{k}(\lambda)$ for $k\in\{0,1\}$.

In general, if $x_{0}$ is an isolated critical point, then $\sigma$ is the
index of $\nabla V(x)$ at $x_{0}$.

If $\eta_{k}(\nu_{k})$ is nonzero, then the equilibrium has a global
bifurcation of periodic solutions starting from the period $2\pi/\nu_{k}$ with
isotropy group $G_{k}$.
\end{theorem}

\begin{proof}
Since $M_{1}(0)$ is a negative number, the sign of the determinant of $M(0)$
is the opposite of $\sigma$. Furthermore, there will be a change of the Morse
number only at values of $\lambda$ where $M_{1}(\lambda)$ is $0$ or where the
self-adjoint matrix $M_{0}(\lambda)$ has one of its two eigenvalues equal to
$0$ ( the other is not $0$, given that the trace is positive). Finally, since
$\lambda=l\nu$, what happens for the fundamental mode ($l=1)$ is reproduced
for higher modes and frequencies which are quotients of the fundamental
frequency by the mode $l$. Here we take the fundamental mode. One is then in
the position of applying Proposition 3.1, p.255 of \cite{IzVi03}, after one
sees the change of orthogonal index. Finally, if $x_{0}$ is an isolated
critical point, then one may perform an orthogonal deformation of the map to
$(\nabla V(x),M(l\nu)x_{l})$, for $l\in\{1,...,p\}$, near $(x_{0},\nu_{k})$,
with a jump at $\nu_{k}$ given by the above formula.
\end{proof}

We say that the bifurcation is \emph{non-admissible} when either: i) the
global branch goes to infinity in norm or period or ii) the branch ends in a
collision path. In any other case we say that the bifurcation is
\emph{admissible}. By global bifurcation we mean either that the bifurcation
is non-admissible or, if the bifurcation is admissible, that the bifurcation
branch returns to other bifurcation points and that the sum of the jumps of
the indices at the bifurcation points, $\eta_{k}(\nu_{k})$, is zero.

\section{Spectral analysis}

Now, we wish to find the bifurcation points of an equilibrium. In order to do
so, we need to analyze the spectrum of the blocks $M_{0}(\lambda)$ and
$M_{1}(\lambda)$. But let us first find the symmetries of the solutions that
bifurcate from these blocks.

For $M_{0}(\lambda)$ we get solutions with isotropy subgroup $\mathbb{Z}_{2}$.
As $\kappa\in\mathbb{Z}_{2}$ has action $\rho(\kappa)x_{0}(t)=Rx_{0}(t)$, this
means that the solutions with symmetry $\mathbb{Z}_{2}$ satisfies
$x_{0}(t)=Rx_{0}(t)$, i.e. $z(t)=0$. Therefore, solutions with symmetry
$\mathbb{Z}_{2}$ are just planar solutions.

For $M_{1}(\lambda)$ we get solutions with isotropy subgroup $\mathbb{\tilde
{Z}}_{2}$. As $(\kappa,\pi)\in\mathbb{\tilde{Z}}_{2}$ has action $\rho
((\kappa,\pi)x_{0}(t)=Rx_{0}(t+\pi)$, then the solutions with symmetry
$\mathbb{\tilde{Z}}_{2}$ satisfy $x_{0}(t)=Rx_{0}(t+\pi)$, i.e.%
\begin{equation}
x(t)=x(t+\pi)\text{, }y(t)=y(t+\pi)\text{ and }z(t)=-z(t+\pi).
\end{equation}
Since the projection of this solution on the $(x,y)$-plane is a $\pi$-periodic
curve, that solution follows twice this planar curve, one time with the
spatial coordinate $z(t)$ and a second time with $-z(t)$. Consequently, there
is at least one $t_{0}$ where $z(t_{0})=z(t_{0}+\pi)=0$. For instance, if only
one $t_{0}$ exists, then the solution looks like a spatial eight near the
equilibrium. For this reason, we will call eight-solutions the solutions with
isotropy subgroup $\mathbb{\tilde{Z}}_{2}$.

\begin{remark}
Actually, the solutions of the satellite are defined in rotating coordinates,
so that the periodic solutions are in general quasiperiodic in fixed coordinates.
\end{remark}

\subsection{Planar solutions}

Let $T$ and $D$ be the trace and determinant of the matrix $M_{0}(0)$. We
point out that the block $M_{0}(0)$ is just the Hessian of the planar
potential at the equilibrium point. In addition, in the first section we have
proven that the trace $T$ is always positive. Now, we want to show that the
bifurcation depends essentially on the sign of $D$.

\begin{proposition}
Let us define $\nu_{\pm}$ as
\[
\nu_{\pm}=\left( 2-T/2\pm\sqrt{(2-T/2)^{2}-D}\right) ^{1/2}.
\]

\begin{description}
\item[(a)] If $D<0$, then $x_{0}$ has a global bifurcation of periodic planar
solutions from $2\pi/\nu_{+}$ with%
\[
\eta_{0}(\nu_{+})=-1.
\]

\item[(b)] If $D>0$, $(2-T/2)^{2}>D$ and $T<4$, then $x_{0}$ has a global
bifurcation of periodic planar solutions from $2\pi/\nu_{+}$ and $2\pi/\nu
_{-}$ with%
\[
\eta_{0}(\nu_{+})=1\text{ and }\eta_{0}(\nu_{-})=-1.
\]

\end{description}
\end{proposition}

\begin{proof}
Since $M_{0}(0)$ is selfadjoint, there is an orthonormal matrix $P\in SO(2)$
such that $M_{0}(0)=P^{T}\Lambda P$, where $\Lambda$ is the eigenvalue matrix
$\mathrm{diag}(\lambda_{1},\lambda_{2})$. Since $M_{0}(\nu)=\nu^{2}%
I-2iJ\nu+M_{0}(0)$ and $J$ commutes with $P$, then%
\[
PM_{0}(\nu)P^{T}=\mathrm{diag}(\nu^{2}+\lambda_{1},\nu^{2}+\lambda_{2}%
)-2\nu(iJ)\text{.}%
\]
In view of $T=\lambda_{1}+\lambda_{2}$ and $D=\lambda_{1}\lambda_{2}$, the
determinant of $M_{0}(\nu)$ is%
\[
\det M_{0}(\nu)=\nu^{4}-2(2-T/2)\nu^{2}+D.
\]
It follows that the determinant has the factorization%
\[
\det M_{0}(\nu)=(\nu^{2}-\nu_{+}^{2})(\nu^{2}-\nu_{-}^{2}).
\]
Consequently, the Morse index of $M_{0}(\nu)$ can change only at $\pm\sqrt
{\nu_{\pm}}$.

For (a), only $\nu_{+}$ is positive, and $\sigma=\mathrm{sgn}(D)=-1$. The
Morse index of $M_{0}(0)$ is $n_{0}(0)=1$ due to $D<0$, and $n_{0}(\infty)=0 $
due to the fact that $M_{0}(\nu)$ has only positive eigenvalues for $\nu$ big
enough. Therefore $\eta_{0}(\nu_{+})=\sigma(1-0)=-1$.

For (b), both numbers $\nu_{\pm}$ are positive, and $\sigma=\mathrm{sgn}(D)=1
$. Moreover, we see that the determinant of $M_{0}(\nu)$ is negative between
$\nu_{-}$ and $\nu_{+}$, thus $n_{0}(\nu)=1$ for $\nu\in(\nu_{-},\nu_{+})$. As
the Morse index at infinity is $n_{0}(\infty)=0 $, we conclude that $\eta
_{0}(\nu_{+})=1-0$. Now, the Morse index of $M_{0}(0)$ is $n_{0}(0)=2$ if
$T<0$ and $n_{0}(0)=0$ if $T>0$. It follows that $\eta_{0}(\nu_{-})=2-1$ if
$T<0$ and $\eta_{0}(\nu_{-})=0-1$ if $T>0.$

Note that this proof is independent of the form of the potential. For the case
of the specific potential of this paper, equation (2) implies that $T>0$.
\end{proof}

\begin{remark}
In the case (b),\ the two local bifurcations can locally collide when the
resonance condition $\nu_{+}=m\nu_{-}$ holds. Moreover,\ it is easy to prove
that the resonance condition is equivalent to%
\[
(4-T)D^{-1/2}\in\{m+m^{-1}:m\in\mathbb{N\}}.
\]

\end{remark}

\begin{remark}
In all other cases different from (a),( b) and $(2-T/2)^{2}=D$, there is no
bifurcation, if $D$ is not $0$, since then the matrix $M_{0}(\lambda)$ is
always invertible. In addition, in the case $(2-T/2)^{2}=D>0$, both points
$\nu_{\pm}$ coincide and $\eta_{0}(\nu_{+})=0$, then we cannot assure or
discard the existence of bifurcation, but probably of a different kind, as
found in \cite{Ba02} and \cite{Si10}. Finally, if $D=0$, then $\nu_{-}=0$ and
$\nu_{+}=(4-T)^{1/2}$, if $T<4$, i.e. $V(x)$ is not a Morse function at
$x_{0}$. In this last case, one may have a bifurcation of relative equilibria
if the masses of the primaries are chosen as parameter and one has a change in
$\sigma$, when one of masses crosses the critical value, or one could have a
secondary bifurcation of periodic solutions if the unfolding has the right
properties, see \cite{I95}. However, in the applications of the present paper,
the potential is a Morse function.
\end{remark}

\begin{remark}
Actually, the satellite equation on the plane is a Hamiltonian system with two
degrees of freedom. We can relate the linear stability of the system with the
bifurcation analysis. Indeed, it can be proven that the equilibrium $x_{0}$ is
linearly stable on the plane if and only if condition (b) is satisfied. Note
that one could argue about the usefulness of a bifurcation result for the
satellite if the arrangement of the primaries is unstable. This is a quite
valid argument from the practical point of view, taking into account the
reality of this model for a problem of mechanics. However, the mathematical
result is independent of the stability of the primaries and furthermore, as
proved in \cite{Ga10} and in an article in preparation, the primaries may
loose their stability and generate stable periodic solutions of the whole
system. In that case, it is much simpler to prove the bifurcation of periodic
solutions for the satellite, assuming, as a first approximation, that the
primaries are at their position of relative equilibrium. Hence, the
mathematical study of the bifurcation is also justified in this framework.

In the case of the Maxwell ring, it is well known that the system of the
primaries is unstable if $n$ is between $3$ and $6$ and the stability is
treated, for $n>6$ and large central mass, in \cite{Ro00}, \cite{VaKo07} and
others. A complete mathematical study of the stability is given in 
\cite{GaIz11}. Thus, if one insists, on physical grounds, that the stability
of the relative equilibrium configuration must be insured in other to have
a study of the bifurcation, one has to restrict to the case $n>6$ and large
central mass, or assume that the primaries are fixed in the rotating frame.

\end{remark}

\begin{remark}
Because there is only one bifurcation value for the frequency in case (a), the
global branch cannot return to the same equilibrium point, so the bifurcation
branch is non-admissible or it is connected to the bifurcation point of
another equilibrium. In fact, if the potential is a Morse function, then one
should get a connection to the small period branch of a minimum (that is with
a jump of $1$). This implies that, in this case, there are at least $n-1$ non
admissible branches starting from saddle points, (see our previous
proposition). In \cite{MaRy04}, one finds other possibilities for branches
starting from a minimum, and $\nu_{+}$, for the restricted three-body problem.
\end{remark}

\subsection{Spatial solutions}

As before, let $\nu_{\pm}$ be the points where $M_{0}(\lambda)$ is not
invertible .

\begin{proposition}
Let us define $\nu_{1}$ as the positive root of
\[
\nu_{1}^{2}=\sum_{j=1}^{n}m_{j}/d_{j}^{\alpha+1}.
\]
Then every equilibrium $x_{0}$ has a global bifurcation of periodic eight
solutions with
\[
\eta_{1}(\nu_{1})=\sigma.
\]
In addition, the local bifurcation branch from $2\pi/\nu_{1}$ is truly
spatial, $z(t)\neq0$, provided the nonresonant condition $\nu_{1}\neq\nu_{\pm
}/2l$ holds.
\end{proposition}

\begin{proof}
It is clear that $M_{1}(\nu)$ is zero only for $\pm\nu_{1}$. Since
$M_{1}(\infty)$ is positive and $M_{1}(0)$ is negative, the Morse indices at
infinity and zero are $n_{1}(\infty)=0$ and $n_{1}(0)=1$. Therefore $\eta
_{1}(\nu_{1})=\sigma(1-0)$. Thus, one has the global bifurcation of periodic
eight solutions.

It remains only to prove that the solutions are truly spatial. In order to
achieve this, we need to prove the nonexistence of solutions of the kind%
\begin{equation}
x(t)=x(t+\pi)\text{, }y(t)=y(t+\pi)\text{ and }z(t)=0 \label{Eq301}%
\end{equation}
near $(x_{0},\nu_{1})$. In fact, the solutions (\ref{Eq301}) are in the fixed
point space of the group $\mathbb{Z}_{2}\times\mathbb{Z}_{2}$ generated by
$\kappa\in\mathbb{Z}_{2}$ and $\pi\in S^{1}$.

Now, the restriction of the derivative of the bifurcation equation to the
fixed point space of $\mathbb{Z}_{2}\times\mathbb{Z}_{2}$ has blocks
$M_{0}(2l\nu_{1})$. Since the matrix $M_{0}(\nu)$ is invertible except for the
points $\nu_{\pm}$, and we suppose $\nu_{\pm}\neq2l\nu_{1}$, the blocks
$M_{0}(2l\nu_{0})$ are invertible. Consequently, the derivative of the
bifurcation equation in the fixed point space of $\mathbb{Z}_{2}%
\times\mathbb{Z}_{2}$ is invertible. Therefore, we get the nonexistence of
planar solutions (\ref{Eq301}) near $(x_{0},\nu_{1})$ from the implicit
function theorem.
\end{proof}

\begin{remark}
Although the nonresonant condition $\nu_{1}\neq\nu_{\pm}/2l$ is sufficient to
assure that the bifurcation from $2\pi/\nu_{1}$ is really spatial, it is not a
necessary condition. If one considers the full three-dimensional problem,
without any special symmetry (except periodicity), then, if one has the
resonance $\nu_{\pm}=2l\nu_{1}$, the jump of orthogonal index has two
components $\eta_{1}(\nu_{1})$ for the fundamental mode and $\eta_{0}(\nu
_{\pm})$ for the $2l$-mode. Since this jump is different from the one caused
by the rescaling of the jump for the solution in the fixed-point subspace of
$\mathbb{Z}_{2}$, which has only the second component for the $2l$-mode, one
obtains a new branch of periodic solutions. In the case of the restricted
three-body problem, this is the branch given in \cite{MaRy04}.
\end{remark}

\section{Applications}

\subsection{A Morse Potential}

We have proven that the potential for the satellite problem has at least $n$
saddle points and a global minimum, provided it is a Morse function.
Consequently, we get the following result:

\begin{theorem}
Each one of the saddle points has a global bifurcation of planar periodic
solutions and a global bifurcation of periodic eight solutions.

Each of the minimum points satisfy one of the following options:\ (a) it has
two global bifurcations of planar periodic solutions and one bifurcation of
periodic eight solutions, or (b) it has one bifurcation of spatial periodic
eight solutions.
\end{theorem}

For the planar bifurcation, each saddle point has a bifurcation with index
$\eta_{0}=-1$ and each minimum point has two bifurcations, if any, one with
$\eta_{0}=1$ and another with $\eta_{0}=-1$. Because an admissible bifurcation
branch has sum of indices $\eta_{0}$ equal to zero, the sum over all
admissible branches is $0$. If $s_{a}$ denotes the number of saddle points
which belong to an admissible branch, $m_{-a}$ the number of minima with jump
of $-1$ which are on an admissible branch and $m_{+a}$ those with jump $1$,
one has that $s_{a}+m_{-a}=m_{+a}$. Let $s_{i}$, $m_{-i}$, $m_{+i}$ be the
numbers of points which are on non-admissble branches and let $s$ be the total
number of saddle points, $m$ the number of minima (including $m_{0}$ those
which are not on any branch), then one gets that $m=m_{0}+m_{-a}+m_{-i}%
=m_{0}+m_{+a}+m_{+i}$ and, since $s=n-1+m$, one has $s_{i}+m_{-i}%
-m_{+i}=n-1+m$, that is the number $s_{i}+m_{-i}$ of points with jump $-1$
belonging to non-admissible branches is at least $n-1+m$. Thus, the number of
points on non-admissible branches is at least the number of saddle points.

Now, since every minimum has a spatial bifurcation with $\eta_{1}=1$ and every
saddle point has a spatial bifurcation with $\eta_{1}=-1$, then a bifurcation
branch of eight solutions is non-admissible or the total number of saddle and
minimum points that it connects is the same and the number of saddle points
which are on non-admissible branches of eight solutions is at least $n-1$.

\subsection{The restricted three-body problem}

In the restricted three-body problem, the primary bodies are at $a_{1}%
=(1-\mu,0)$ and $a_{2}=(-\mu,0)$ with masses $m_{1}=\mu$ and $m_{2}=1-\mu$.
Hence, the potential of the satellite is%
\[
V(x)=\frac{1}{2}\left\Vert \bar{I}x\right\Vert ^{2}+\sum_{j=1}^{2}m_{j}%
\phi_{\alpha}(\left\Vert x-(a_{j},0)\right\Vert ).
\]

This problem is well known on the plane, see for instance \cite{MeHa91}. There
are only five equilibrium points called Lagrangians. Two of these equilibrium
points form an equilateral triangle with the primary bodies $a_{1}$ and
$a_{2}$, and they are minima of the planar potential. Three of the equilibrium
points are collinear with the primaries, also called Eulerian points, and they
are saddle points of the potential. All of these relative equilibria are
non-degenerate, that is $V(x)$ is a Morse function.

Also, it is well known that the minimum points have two bifurcation
frequencies $\nu_{\pm}$ for $\mu<\mu_{1}$, where $\mu_{1}=(1-(\alpha
+1)^{-1}\sqrt{\alpha(30-\alpha)-33)/12})/2$, when $\alpha$ is in the
interval$(15-8\sqrt{3},3)$, is the critical Routh ratio and without any
restriction on $\mu$ if $\alpha$ belongs to the interval $(1, 15-8\sqrt{3})$.
This comes from the fact that the trace $T=\alpha+1$ and the determinant $D=
3(\alpha+1)^{2} \mu(1-\mu)/4$, with the conditions $T<4$ and $(2-T/2)^{2}>D$.
In that case,%
\[
\nu_{\pm}^{2}=\left( 3-\alpha\pm\sqrt{(3-\alpha)^{2}-3(\alpha+1)^{2}\mu
(1-\mu)}\right) /2\text{.}%
\]

Note that $\nu_{+}/\nu_{-}$ tends to infinity when $\mu$ tends to $0$, thus
there is an infinite number of resonance values for $\mu$, when $\mu$ goes to
$0$.

For the saddle points we have only the bifurcation point $\nu_{+}$, where
\[
\nu_{+}^{2}=1- (\alpha-1)\nu_{1}^{2}/2+((\alpha+1)^{2}\nu_{1}^{4}%
/4-2(\alpha-1)\nu_{1}^{2})^{1/2}%
\]
with $\nu_{1}^{2}=\sum_{j=1}^{2}m_{j}/d_{j}^{\alpha+1}>1$, since $D^{2}%
V(x_{0}) = diag(1+\alpha\nu_{1}^{2}, 1-\nu_{1}^{2})$.

Consequently, we get the classical global bifurcation of planar periodic
solutions, with at least three equilibria on non-admissible branches, see
\cite{MaRy04} for the case $\alpha=2$.

Now, we wish to find bifurcation of spatial periodic eight-solutions.

\begin{theorem}
In the restricted three-body problem each one of the five equilibria has a
global bifurcation of spatial periodic eight-solutions.
\end{theorem}

\begin{proof}
We only need to prove the nonresonant condition $\nu_{1}>\nu_{\pm}/2l$ at
equilibrium points. For the triangular Lagrangian points we have that $\nu
_{1}=1$ and $\nu_{\pm}\in(0,\sqrt{(3-\alpha)}$ for $\mu\in(0,1)$, therefore
$\nu_{1}>\nu_{\pm}/2l$.

For the collinear Lagrangian points, since $\nu_{+}$ is given in terms of
$\nu_{1}$, we need to prove that $4l^{2}\nu_{1}^{2}\neq\nu_{+}^{2}$, or
equivalently $(\alpha+1)^{2}\nu_{1}^{4}/4-2(\alpha-1)\nu_{1}^{2}\neq
((4l^{2}+(\alpha-1)/2)\nu_{1}^{2}-1)^{2}$. The last inequality is also
equivalent to $a\nu_{1}^{4}-2b\nu_{1}^{2}+1\neq0$, where $a=\left(
4l^{2}+(\alpha-1)/2\right) ^{2}-(\alpha+1)^{2}/4$ and $b=4l^{2}-(\alpha
-1)/2$. But since $b^{2}-a=(\alpha+1)^{2}/4-8(\alpha-1)l^{2}<0$ is satisfied
for all $l\geq1$, if $\alpha\in(15-8\sqrt{3},3)$, then the quadratic equation
$a\nu_{1}^{4}-2b\nu_{1}^{2}+1=0$ does not have solutions and $4l^{2}\nu
_{1}^{2}\neq\nu_{+}^{2}$. On the other hand, if $\alpha\in(1,15-8\sqrt{3})$
and $l=1$, then the quadratic equation has its largest root less than $1$,
which contradicts the fact that at the saddle point $\nu_{1}>1$. Thus, there
is no resonance and the branch is truly spatial and at least one branch is non-admissible.
\end{proof}

\subsection{The Maxwell's Saturn ring}

In this section, we analyze the satellite problem when the primaries form a
polygonal relative equilibrium. Hereafter, we identify the real and complex planes.

The polygon consists of one body of mass $\mu$ at $a_{0}=0$, and $n$ bodies of
mass $1$ at each vertex of a regular polygon, for instance $a_{j}=ae^{ij\zeta
}$ for $j\in$ $\{1,...,n\}$, where $\zeta=2\pi/n$. It is easy to prove that
the positions $a_{j}$ form a relative equilibrium provided that $a^{\alpha
+1}=s+\mu$, where $s$ is defined by%
\[
s=\frac{1}{2^{\alpha}}\sum_{j=1}^{n-1}\frac{1}{\sin^{\alpha-1}(j\zeta
/2)}\text{.}%
\]

Moreover, we can make the change of variable $x=au$ in such a way that the
equation is $\ddot{u}+2\bar{J}\dot{u}=\nabla V(u)$ with the potential%
\[
V(u)=\frac{1}{2}\left\Vert \bar{I}u\right\Vert ^{2}+\sum_{j=1}^{n}\frac
{1}{s+\mu}\phi_{\alpha}(\left\Vert u-(e^{ij\zeta},0)\right\Vert )+\frac{\mu
}{s+\mu}\phi_{\alpha}(\left\Vert u\right\Vert )\text{.}%
\]
Now we point out that the case $n=2$ with $\mu=0$ is just a particular case of
the restricted three-body problem, hence we shall analyze only the cases $n=2$
with $\mu>0$ and $n\geq3$ with $\mu\geq0$.

\subsubsection*{Existence of equilibria}

Remember that all equilibrium points of the satellite are in the plane. So, we
assume, for this purpose, that the satellite is in the plane, i.e. the
potential is%
\[
V(u)=\frac{1}{2}\left\Vert u\right\Vert ^{2}+\sum_{j=1}^{n}\frac{1}{s+\mu}%
\phi_{\alpha}(\left\Vert u-e^{ij\zeta}\right\Vert )+\frac{\mu}{s+\mu}%
\phi_{\alpha}(\left\Vert u\right\Vert )
\]
with $u\in\mathbb{R}^{2}$.

\begin{proposition}
For $\mu=0$, the origin $u_{0}=0$ is a critical point. In addition, we have
for $n\geq3$ that $D^{2}V(0)=\lambda I$ with $\lambda>0.$
\end{proposition}

\begin{proof}
That the origin is a critical point follows from the fact that%
\[
\nabla_{u}V(0)=\frac{1}{s}\sum_{j=1}^{n}e^{ij\zeta}=0\text{.}%
\]
Now, since $D^{2}V(0)$ has real eigenvalues and $D^{2}V(0)$ is $D_{n}%
$-equivariant, by Schur's lemma we have $D^{2}V(0)=\lambda I$ for $n\geq3$.
That $\lambda>0$ is due to the fact that the trace $T=2\lambda$ is always positive.
\end{proof}

Now for $u\neq0$, we may simplify the analysis if we change to polar
coordinates. For these coordinates the potential is%
\[
V(r,\varphi)=r^{2}/2+\frac{\mu}{s+\mu}\phi_{\alpha}(\left\Vert r\right\Vert
)+\sum_{j=1}^{n}\frac{1}{s+\mu}\phi_{\alpha}(\left\Vert r-e^{i(j\zeta
-\varphi)}\right\Vert ).
\]

Let us observe that the potential $V$ is $D_{n}$-invariant for the action
$\rho(\zeta)u=e^{i\zeta}u$ and $\rho(\kappa)u=\bar{u}$, thus, critical points
will be $D_{n}$-orbits of points. It follows that the potential $V(r,\varphi)
$ is even and $2\pi/n$-periodic in $\varphi$, hence, we may restrict our
analysis to points with $\varphi\in\lbrack0,\pi/n]$.

Now, we will show that the potential has three orbits of critical points. To
achieve this goal, we need first to prove the following lemma.

\begin{lemma}
For $n\geq3$, the derivative $V_{r}$ at $e^{i\pi/n}$ is negative,%
\[
V_{r}(1,\pi/n)<0.
\]

\end{lemma}

\begin{proof}
The derivative of $V(r,\varphi)$ is%
\begin{equation}
V_{r}(r,\varphi)=r-\frac{\mu}{s+\mu}\frac{1}{r^{\alpha}}-\frac{1}{s+\mu}%
\sum_{j=1}^{n}\frac{r-\cos(j\zeta-\varphi)}{\left\Vert r-e^{i(j\zeta-\varphi
)}\right\Vert ^{\alpha+1}}. \label{Eq431}%
\end{equation}
Therefore, at $e^{i\pi/n}$, we have%
\[
V_{r}(1,\pi/n)=\frac{s}{s+\mu}-\frac{1}{s+\mu}\left( \sum_{j=1}^{n}\frac
{1}{2^{\alpha}}\frac{1}{\sin^{\alpha-1}(j-1/2)\zeta/2}\right) =\frac
{s-\sigma}{s+\mu}\text{,}%
\]
where $\sigma$ is the sum between parentheses.

So it remains to prove that $s<\sigma$. In order to do so, we need some
inequalities. Since $n\geq3$, we have the first inequality%
\[
2^{\alpha}s=\sum_{j=1}^{n-1}\frac{1}{\sin^{\alpha-1}(j\zeta/2)}\leq2\sum
_{j\in\lbrack1,n/2]\cap\mathbb{N}}\frac{1}{\sin^{\alpha-1}(j\zeta/2)},
\]
where equality holds for $n$ odd. Similarly, we have the second inequality%
\[
2^{\alpha}\sigma=\sum_{j=1}^{n}\frac{1}{\sin^{\alpha-1}(j-1/2)\zeta/2}%
\geq2\sum_{j\in\lbrack1,n/2]\cap\mathbb{N}}\frac{1}{\sin^{\alpha
-1}(j-1/2)\zeta/2}\text{,}%
\]
where equality holds for $n$ even. Finally, since $\sin(j-1/2)\zeta/2<\sin
j\zeta/2$ for $j\in\lbrack1,n/2]$, then we have the third inequality%
\[
\frac{1}{\sin^{\alpha-1}(j-1/2)\zeta/2}>\frac{1}{\sin^{\alpha-1}(j\zeta
/2)}\text{.}%
\]
The fact $\sigma>s$ follows from these inequalities.
\end{proof}

In \cite{BaEl03}, one may find an integral representation which is used to
prove the next corollary. In addition, a direct proof of the integral
representation and of this corollary will be given in the last section.

\begin{corollary}
For $\alpha\in(1,3)$, the derivative $V_{r}(r,\varphi)$ is the product of
$-\sin(n\varphi)$ with a positive function $\omega(r,\varphi)$,
\[
V_{\varphi}(r,\varphi)=-\sin(n\varphi)\omega(r,\varphi)\text{.}%
\]

\end{corollary}

We may now prove the existence of $\mathbb{Z}_{n}$-orbits of equilibrium points.

\begin{proposition}
For $\alpha\in(1,3)$ and $n\geq3$ there are three orbits of critical points.
We are showing only the points of the $\mathbb{Z}_{n}$-orbits with $\varphi
\in\lbrack0,\pi/n]$:

\begin{description}
\item[(a)] If $\mu\in(0,\infty)$, there are two saddle points at $r_{2}$ and
$r_{1}$, with $0<r_{2}<1<r_{1}$, and there is a minimum point at $r_{3}%
e^{i\pi/n}$, with $r_{3}>1$.

\item[(b)] If $\mu=0$, there are two saddle points at $r_{1}$ and
$r_{2}e^{i\pi/n}$, with $0<r_{2}<1<r_{1}$, and there is a minimum point at
$r_{3}e^{i\pi/n}$, with $r_{3}>1$.
\end{description}

Furthermore, there are no other critical points when $\varphi\in\lbrack
0,\pi/n)$.
\end{proposition}

\begin{proof}
Since $V_{\varphi}(r,\varphi)=-\sin(n\varphi)\omega(r,\varphi)$, with a
positive function $\omega(r,\varphi)$, then $V_{\varphi}(r,\varphi)=0$ only
for $\varphi=k\pi/n$. Furthermore, at these points we have $V_{\varphi\varphi
}(r,k\pi/n)=-n\omega(r,\varphi)\cos k\pi$. Consequently, the critical points
must be in $\varphi\in\{0,\pi/n\}$ with%
\[
V_{\varphi\varphi}(r,0)<0\text{ and }V_{\varphi\varphi}(r,\pi/n)>0.
\]
Thus, in order to find critical points, we need to look only for points where
$V_{r}(r,\varphi)=0$, with $\varphi=0,\pi/n$.

Before we start finding critical points, we wish to prove that all the
critical points\ with $\varphi=0$ are saddle points. The trace of
$D^{2}V(x_{0})$ at a critical point is
\[
T=V_{xx}+V_{yy}=V_{rr}+r^{-2}V_{\varphi\varphi}.
\]
Similarly, it is easy to see that the determinant of $D^{2}V(x_{0})$ at a
critical point is
\[
D= V_{rr}V_{\varphi\varphi}r^{-2}.
\]
Now, since $T$ is always positive and $V_{\varphi\varphi}(r,0)$ is always
negative, then $V_{rr}(r,0)$ is positive. Consequently, all critical points,
with $\varphi=0$, satisfy
\[
V_{rr}(r,0)>0\text{ and }V_{\varphi\varphi}(r,0)<0.
\]

For $\mu\in\lbrack0,\infty)$, the potential $V(r,0)$ goes to infinity when
$r\rightarrow\{1,\infty\}$. Hence, the potential has a saddle point at
$r_{1}\in(1,\infty)$. Now, if there were another critical point $r_{\ast}$ in
$(1,\infty)$, then $V_{rr}(r_{\ast},0)$ would be positive.\ In that case there
would be another critical point between $r_{1}$ and $r_{\ast}$ with
$V_{rr}(r,0)\leq0$. But that cannot happen, and consequently $r_{1}$\ is
unique in $(1,\infty)$.

For $\mu\in(0,\infty)$, the potential $V(r,0)$ goes to infinity when
$r\rightarrow\{1,0\}$. Hence the potential has a saddle point with $r_{2}%
\in(0,1)$. As before with $r_{1}$, we can prove that $r_{2}$ is unique in
$(0,1) $.

For $\mu=0$, remember that $V_{r}(0,\varphi)=0$ and $V_{rr}(0,\varphi)>0$.
Then, by a similar argument to the uniqueness of $r_{1}$ we can prove that the
potential $V(r,0)$ does not have critical points in $(0,1)$. Now, for
$\varphi=\pi/n$, since $V_{r}(1,\pi/n)$ is negative and $V_{r}(0,\pi/n)=0$
with $V_{rr}(0,\pi/n)>0$, there must be a $r_{2}<1$ such that $V_{r}(r_{2}%
,\pi/n)=0$ with $V_{rr}(r_{2},\pi/n)<0$. Consequently $r_{2}e^{i\pi/n}$ is a
saddle point.

For $\mu\in\lbrack0,\infty)$, since $V_{r}(1,\pi/n)$ is negative and since
$V_{r}(r,\pi/n)$ goes to infinity as $r\rightarrow\infty$, there is a critical
point $r_{3}\in(1,\infty)$ such that $V_{rr}(r_{3},\pi/n)>0$. Therefore
$r_{3}e^{i\pi/n}$\ is a minimum.
\end{proof}

In the article \cite{BaEl04}, the existence of these three orbits of
equilibrium points is proven, as well as their stability. However, our proofs
are simpler.

For $n=2$ and $\mu>0$ we can prove the previous proposition with the same
argument, except for the existence of $r_{3}$. Instead, we get the existence
of a $r_{3}\in(0,\infty)$ because the potential $V(r,\pi/2)$ goes to infinity
when $r\rightarrow0,\infty$.

Now, in \cite{BaEl04}, the question of the existence of more critical points
was left open. Actually, for $n=2$ and $\mu>0$ we can prove the following:

\begin{proposition}
For $n=2$ and $\mu>0$ the previous proposition is true and there are no other
critical points.
\end{proposition}

\begin{proof}
It remains only to prove that $r_{3}$ is in $(1,\infty)$ and is unique. Let us
define $f(r)=-2(r^{2}+1)^{-(\alpha+1)/2}$. After some computations, we find
that the derivative $V_{r}(r,\pi/2)$ satisfies the equality%
\begin{equation}
(s+\mu)V_{r}=r(f(r)+s)+\mu(r-r^{-\alpha}). \label{Eq432}%
\end{equation}
Let us denote the $\mu$-dependence of the potential as $V(r,\varphi;\mu)$.
Therefore, from the equality (\ref{Eq432}), we have that $V_{r}(r,\pi
/2;\mu)<V_{r}(r,\pi/2;0)$ for $r\leq1$. Now, as the three body problem is the
case $n=2$ with $\mu=0$, we know that $V_{r}(r,\pi/2;0)=0$ only at the
triangular Lagrangian point $r=\sqrt{2}$. Furthermore, $V_{r}(r,\pi/2;0)<0$
for $r\leq1$, and hence $V_{r}(r,\pi/2;\mu)<0$ for $r\leq1$.

Now, let us analyze the case $r>1$. From (\ref{Eq432}), we see, for the second
derivative, that%
\[
(s+\mu)V_{rr}=(rf^{\prime}+f)+s+\mu(1+\alpha r^{-(\alpha+1)})\text{.}%
\]
Since $rf^{\prime}+f=2\left( r^{2}\alpha-1\right) \left( r^{2}+1\right)
^{(\alpha+3)/2}$ is a positive function, then $V_{rr}(r,\pi/2)>0$ for $r>1$.
From this statement, we conclude that $V_{r}(r,\pi/2)$ has only the critical
point $r_{3}\in(1,\infty)$.
\end{proof}

We proved that there may be more critical points only if $\varphi=\pi/n$. And
indeed, for $n\geq3$ we can find more critical points when $\mu$ is near zero.

\begin{proposition}
For $n\geq3$ and $\mu$ near zero the potential has also a minimum and a saddle
point at $r_{4}e^{i\pi/n}$ and $r_{5}e^{i\pi/n}$ with $r_{4}<r_{5}<1 $. On the
other hand, for $\mu$ large, $r_{3}e^{i\pi/n}$ is the only critical point on
that line.
\end{proposition}

\begin{proof}
As before, we represent the dependence of the potential in $\mu$ as
$V_{r}(r,\pi/n;\mu)$. Remember that $V_{r}(0,\varphi;0)=0$ with $V_{rr}%
(0,\varphi;0)=\lambda>0$ for $n\geq3$, then there is a $r_{\ast}%
\in(0,\varepsilon)$ such that $V_{r}(r_{\ast},\pi/n;0)>0$. Therefore,
$V_{r}(r_{\ast},\pi/n;\mu)>0$ for $\mu$ near zero due to the continuity.
Gathering data, we get $V_{r}(0,\pi/n)=-\infty$, $V_{r}(r_{\ast},\pi/n)>0 $
and $V_{r}(1,\pi/n)<0$ for $\mu$ near zero. Consequently, there are two points
$r_{4}$ and $r_{5}$ where $V_{r}(r,\pi/n)$ is zero with $r_{4}<r_{5}<1$.
Moreover, the second derivative satisfies $V_{rr}(r,\pi/n)\geq0$ for $r$ close
to $r_{4}$ and $V_{rr}(r,\pi/n)\leq0$ for $r$ close to $r_{5}$. Therefore,
$r_{4}e^{i\pi/n}$ is a minimum and $r_{5}e^{i\pi/n}$ is a saddle point. On the
other hand, for $\mu$ large it is easy to see that $V_{r}$ is strictly increasing.
\end{proof}

The existence of the solutions $r_{4}e^{i\pi/n}$ and $r_{5}e^{i\pi/n}$ was
pointed out in the paper \cite{ArEl04}.

\subsubsection*{Existence of bifurcation}

At the saddle points we have the following result:

\begin{theorem}
The potential has two $\mathbb{Z}_{n}$-orbits of saddle points for $n\geq2$,
and one more when $n\geq3$ and $\mu$ is near zero. Furthermore, each one of
the saddle points has one global bifurcation of planar periodic and one
bifurcation of periodic eight solutions.
\end{theorem}

\begin{proof}
The saddle point on the line $\varphi=0$ is non-degenerate, while the critical
points on the line $\varphi=\pi/n$ are isolated, since $V_{r}$ is locally
analytic. Hence the index at $r_{5}$ will be $-1$, unless $r_{5}$ and $r_{4}$
coincide, in which case the index would be $0$.
\end{proof}

Also at the orbit of minimum points we have the following:

\begin{theorem}
The potential has one $\mathbb{Z}_{n}$-orbit of minimum points for $n\geq2$,
and one more when $n\geq3$ and $\mu$ is near zero. Moreover, provided $\mu$ is
big enough, each minimum point has two global bifurcations of planar periodic
solutions and one global bifurcation of periodic eight solutions. On the other
hand, if $\alpha\geq2$ and $\mu$ is small, the minimum $r_{4}e^{i\pi/n}$ has
no bifurcation of planar periodic solutions and it has a global bifurcation of
spatial eight solutions.
\end{theorem}

\begin{proof}
Since the minima are isolated, with $\sigma=1$, we only need to confirm that
the bifurcation condition (b), $T<4$ and $(2-T/2)^{2}>D>0$, is satisfied at
$r_{3}e^{i\pi/n}$ provided that $\mu$ is big enough.

As $r_{3}$ is a critical point, i.e. $V_{r}(r_{3}e^{i\pi/n};\mu)=0$, from
(\ref{Eq431}) we can see that $r_{3}(\mu)\rightarrow1$ when $\mu
\rightarrow\infty$. From the definition (\ref{Eq101}) of $A_{j}$, the matrix
\[
M_{0}(0)=I+\frac{1}{s+\mu}\sum_{j=1}^{n}A_{j}+\frac{\mu}{s+\mu}A_{0}%
\]
converges, when $\mu\rightarrow\infty$, to the matrix%
\[
I+A_{0}=(\alpha+1)\left(
\begin{array}
[c]{cc}%
(\cos\pi/n)^{2} & \cos\pi/n\sin\pi/n\\
\cos\pi/n\sin\pi/n & (\sin\pi/n)^{2}%
\end{array}
\right) \text{.}%
\]
Given that $T(\mu)\rightarrow\alpha+1$ and $D(\mu)\rightarrow0$ when
$\mu\rightarrow\infty$, then $(2-T/2)^{2}-D\rightarrow\varepsilon>0$ for
$\alpha\in(1,3)$. Consequently, for $\alpha\in(1,3)$, at the minimum point the
bifurcation condition (b) holds provided $\mu$ is big enough. Finally, for the
minima inside the unit disc, one has that $d_{1}$ and $d_{n}$ are less than
$1$, hence, for $\alpha\geq2$ one has that $T>4$.
\end{proof}

\begin{remark}
As a consequence of the previous proposition, we get that the minimum point
$r_{3}e^{i\pi/n}$ is linearly stable for $\mu$ big enough. This is one of the
aims of the article \cite{BaEl04} where the stability, for the system of the
primaries and the satellite, is proven for $n\geq7$ and $\mu$ big enough.
\end{remark}

\begin{remark}
For $n\geq3$ with $\mu=0$, as we have seen before, at the origin $x_{0}=0$, we
have $M_{0}(0)=\lambda I$. Actually, since the trace $T>2$, we can prove that
the condition for bifurcation (b) is not satisfied.\ Hence the origin is a
minimum point without bifurcation of planar periodic solutions.

On the other hand, the origin is a minimum with one bifurcation of spatial
eight periodic solutions. Moreover, we can prove, from the symmetries, that
the bifurcating solutions satisfy $x(t)=0$, $y(t)=0$ and $z(t)=-z(t+\pi)$. In
fact, we can find $z(t)$ by quadrature from the equation $\ddot{z}=\nabla
V(z)$, with $V(z):=\frac{n}{\nu^{2}}\phi_{\alpha}(\sqrt{z^{2}+1})$, with
$\nu^{2}$ close to $n$. Recall that, at it is well known, that in this case
the system of the primaries is linearly unstable.
\end{remark}

\begin{remark}
The study of the bifurcation of periodic solutions, in the plane and also in
space, for the full system of primaries, will be published in another paper.
\end{remark}

\section{ Appendix: Integral representation}

Let us define the sum $S(r,\varphi)$\ as%
\[
S(r,\varphi)=\sum_{j=1}^{n}\frac{1}{\left\Vert r-e^{i(j\zeta-\varphi
)}\right\Vert ^{\beta}}\text{.}%
\]
In \cite{BaEl03} the following integral representation of $S(r,\varphi)$ is
proved. We shall give here a direct proof using Cauchy integrals.

\begin{lemma}
For $\beta\in(0,2)$\ and $r\in(0,1)$\ the function $S(r,\varphi)$ has the
integral representation
\begin{align*}
S(r,\varphi) & =\frac{n}{\pi}\sin(\pi\beta/2)\int_{0}^{1}\frac{1}{(\tau
^{-1}-1)^{\beta/2}}f(\tau)d\tau\text{ with}\\
f(\tau) & =\frac{1}{\tau(1-r^{2}\tau)^{\beta/2}}\frac{1-(r\tau)^{2n}%
}{\left\Vert 1-(\tau re^{-i\varphi})^{n}\right\Vert ^{2}}\text{.}%
\end{align*}

\end{lemma}

We are now in a position of proving the corollary on $V_{\varphi}$.

From the integral representation we get that%
\[
S_{\varphi}=-\sin(n\varphi)\left( \frac{n^{2}}{\pi}\sin\frac{\pi\beta}{2}%
\int_{0}^{1}\frac{1}{(\tau^{-1}-1)^{\beta/2}}\frac{2(r\tau)^{n}}{\tau
(1-r^{2}\tau)^{\beta/2}}\frac{1-(r\tau)^{2n}}{\left\Vert 1-(\tau
re^{-i\varphi})^{n}\right\Vert ^{4}}d\tau\right)
\]
for $r\in\lbrack0,1)$. Therefore $S_{\varphi}(r,\varphi)$ is the product of
$-\sin(n\varphi)$ with the function between parentheses, which is positive.
Moreover, since the sum $S(r)$\ satisfies the equality $S(1/r)=r^{\beta}S(r)
$, we conclude that $S_{\varphi}(r,\varphi)$ is the product of $-\sin
(n\varphi)$ with a positive function for $\beta\in(0,2)$ and $r\neq1$.

We use this integral representation to prove the following: Set $\beta
=\alpha-1$, then we have $\phi_{\alpha}(r)=1/(\beta r^{\beta})$. Now, we can
express the potential $V(r,\varphi)$ in terms of $S(r,\varphi)$ as%
\[
V(r,\varphi)=r^{2}/2+\frac{\mu}{s+\mu}\phi_{\alpha}(r)+\frac{1}{s+\mu}\frac
{1}{\beta}S(r,\varphi)\text{.}%
\]
Since $V$ depends on $\varphi$ only through $S(r,\varphi)$, we conclude that
$V_{\varphi}(r,\varphi)$ is the product of $-\sin(n\varphi)$ with a positive
function for $\alpha=\beta+1\in(1,3)$.

We may now prove the last lemma:

\begin{proof}
Let us define the function $w(z)$ as
\[
w(z)=\frac{1}{[z^{-1}-1]^{\beta/2}}\text{.}%
\]
This function has an analytic extension to $\mathbb{C}-[0,1]$. Indeed, using
the principal branch of the logarithm we can extend it as%
\[
w(z)=e^{-(\beta/2)[\log\left\vert z^{-1}-1\right\vert +i\arg(z^{-1}%
-1)]}\text{.}%
\]
Let $w^{\pm}(r)$ be the limits $w^{\pm}(r)=\lim_{\varepsilon\rightarrow
0}w(r\pm i\left\vert \varepsilon\right\vert )$ for $r$ $\in(0,1)$, then%
\[
w^{+}(r)=e^{-i\beta\pi}\frac{1}{(r^{-1}-1)^{\beta/2}}\text{ and }%
w^{-}(r)=\frac{1}{(r^{-1}-1)^{\beta/2}}\text{.}%
\]

Let $\Omega_{\varepsilon}$ be the set of points%
\[
\Omega_{\varepsilon}=\{\left\vert z\right\vert <1/\varepsilon:\text{
}\left\vert z-r\right\vert >\varepsilon\text{ for }r\in\lbrack0,1]\}\text{.}%
\]
As the function $w(z)f(z)$ is of order $O(1/z^{1+\beta/2})$ when
$z\rightarrow\infty$, if $\beta>0$ then the integral over the circle of radius
$1/\varepsilon$ goes to zero when $\varepsilon\rightarrow0$. Moreover, since
the product $w(z)f(z)$ is of order $O(z^{\beta/2-1})$ when $z\rightarrow0$ and
of order $O((1-z)^{-\beta/2})$ when $z\rightarrow1$, then for $\beta<2$ the
integrals over the half circles around $z=0$ and $z=1$ go to zero when
$\varepsilon\rightarrow0$. Consequently, we have that%
\begin{align*}
\lim_{\varepsilon\rightarrow0}\int_{\partial\Omega_{\varepsilon}}w(z)f(z)dz
& =\int_{0}^{1}[w^{+}(\tau)-w^{-}(\tau)]f(\tau)d\tau\\
& =(e^{-i\beta\pi}-1)\int_{0}^{1}w^{-}(\tau)f(\tau)d\tau\text{.}%
\end{align*}

Now, the function $w(z)f(z)$ has $n$ poles in $\mathbb{C}-[0,1]$ and another
one at $z=r^{-2}$, but the residue at $z=r^{-2}$ is zero because $\beta
/2\in(0,1)$. The other $n$ poles are the roots of the polynomial function%
\[
g(z)=\left\Vert 1-(zre^{-i\varphi})^{n}\right\Vert ^{2}=1+(rz)^{2n}%
-2(rz)^{n}\cos n\varphi\text{.}%
\]
Consequently, the poles are found at the points $z_{j}^{-1}=(re^{-i\varphi
})e^{ij\zeta}$ for $j=0,...,n-1$. As $(rz_{j})^{n}=e^{in\varphi}$, the
derivative of $g$ at the pole $z_{j}$ is
\[
g^{\prime}(z_{j})=2nz_{j}^{-1}e^{in\varphi}(e^{in\varphi}-\cos n\varphi
)=2inz_{j}^{-1}e^{in\varphi}\sin n\varphi.
\]
Consequently, the residue of $w(z)f(z)$ at the pole $z_{j}$ is%
\[
\mathrm{res}_{z_{j}}w(z)f(z)=\frac{1}{[(z_{j}^{-1}-1)(1-r^{2}z_{j})]^{\beta
/2}}\frac{1-e^{2ni\varphi}}{z_{j}g^{\prime}(z_{j})}\text{.}%
\]

Moreover, since $r^{2}z_{j}=\bar{z}_{j}^{-1}$ and $(1-e^{2ni\varphi}%
)/(z_{j}g^{\prime}(z_{j}))=-1/n$, then%
\[
\mathrm{res}_{z_{j}}w(z)f(z)=-\frac{1}{n}\frac{1}{(-1)^{\beta/2}}\frac
{1}{\left\Vert z_{j}^{-1}-1\right\Vert ^{\beta}}=-\frac{1}{n}e^{-i\pi\beta
/2}\frac{1}{\left\Vert r-e^{i(j\zeta-\varphi)}\right\Vert ^{\beta}}\text{.}%
\]

Now, from the Cauchy theorem, we obtain that%
\[
\lim_{\varepsilon\rightarrow0}\int_{\partial\Omega_{\varepsilon}%
}w(z)f(z)dz=2\pi i\sum_{z\in\mathbb{C}-[0,1]}\mathrm{res}_{z}w(z)f(z)\text{.}%
\]
Consequently, from the integral and the residues we have%
\[
(e^{-i\beta\pi}-1)\int_{0}^{1}\frac{1}{(\tau^{-1}-1)^{\beta/2}}f(\tau
)d\tau=-2\pi ie^{-i\pi\beta/2}\frac{1}{n}\sum_{j=1}^{n}\frac{1}{\left\Vert
r-e^{i(j\zeta-\varphi)}\right\Vert ^{\beta}}\text{.}%
\]
Finally, we conclude that%
\[
\sum_{j=1}^{n}\frac{1}{\left\Vert r-e^{i(j\zeta-\varphi)}\right\Vert ^{\beta}%
}=\frac{n}{\pi}\sin(\pi\beta/2)\int_{0}^{1}\frac{1}{(\tau^{-1}-1)^{\beta/2}%
}f(\tau)d\tau.
\]

\end{proof}

\section{Conclusion}

For an arbitrary relative equilibrium of primaries in the plane, we have
proved that each relative equilibrium of the satellite generates several
global branches of periodic solutions: for a saddle point one gets a global
branch of planar solutions and a global branch of eight-solutions, which are
truly spatial if a non-resonance condition holds. For a minimum point of the
potential, one gets either two global branches of planar solutions (long and
short period) and a global branch of eight-solutions, or only the branch of
eight-solutions which is then truly spatial.

A global branch may be non-admissible if the period or the norm of the
solutions on the branch go to infinity or the branch goes to collision with
one of the primaries. On the other hand, if the branch is admissible, then the
sum of the jumps of the Morse indices at the critical points on the branch
must be zero. In particular, a saddle point has to be connected with a short
period minimum, the number of points on non-admissible planar branches is at
least the number of saddle points and the number of saddle points on these
non-admissible planar branches is at least one less than the number of
primaries. Also, the number of saddle points on non-admissible branches of
eight-solutions is at least one less than the number of primaries.

We have applied this general result in order to describe a rather complete
picture of the restricted three-body problem and of the restricted Maxwell ring.

The topological degree approach, combined with the use of the orthogonality
(or first integrals) and a systematic use of representation theory, gives
information which is a good complement to classical analytical local calculus
and allows flexible applications. In particular, one may extend easily these
results to different potentials and to systems with more bodies.

For concrete situations, there are many local techniques, such as normal form
theory which often requires to check some generic assumptions ( this is not
always done in practice), Poincar\'{e} mappings, stable and unstable manifold
decomposition of the phase space and so on. For a low dimensional bifurcation
equation, there is a common starting point for these analytical methods and
for the computation of a topological degree, that is the linearization of the
equations. Higher order approximations may give a better local picture of the
bifurcated solutions. But, as soon as there are resonances or more couplings,
the analytical methods become more difficult to apply, while the topological
degree approach can still give a complementary information on the set of
bifurcating solutions, in particular on the global properties of the branches.
It is important to point out that, in many relevant applied problems, one may
carry out symbolic manipulations of high order which may be even converted
into a valid mathematical proof using interval arithmetics. We are fully
familiar with higher order symbolic manipulations of formal power series and
the use of computer assisted proofs.

With these considerations in mind, we have several papers in preparation on
bifurcation of the whole arrangement of primaries, either as relative
equilibria or as periodic solutions. For instance, in the case of the Maxwell
ring, one gets $n$ global branches of periodic solutions, each with different
symmetries and where the central mass plays an important role, for the
existence of these periodic solutions. Similar results were obtained for
vortices, filaments, charged particles and nonlinear oscillators. See
\cite{Ga10}.

\section{Acknowledgements}

The authors wish to thank the referees for their comments and for pointing out
some references. Also, C.G-A wishes to thank the CONACyT for his scholarship
and J.I for the grant No. 133036.

\end{document}